\documentclass{amsart}
\usepackage{hyperref}
\usepackage[ansinew]{inputenc}
\usepackage{graphicx}
\usepackage{amsmath}
\usepackage{amsthm}
\usepackage{amssymb,bbm, color}%
\usepackage[numbers, square]{natbib}
\usepackage{mathrsfs}

\newcommand{\bL}{\mathbb{L}}

\newcommand{\Herm}{\mathrm{He}}

\newcommand{\fD}{\mathscr{D}}

\DeclareMathOperator*{\argmin}{arg\,min}

\newcommand{\R}{\mathbb{R}}
\newcommand{\N}{\mathbb{N}}

\newcommand{\C}{\mathbb{C}}
\newcommand{\Z}{\mathbb{Z}}

\renewcommand{\P}{\mathbb{P}}

\renewcommand{\Im}{\operatorname{Im}}

\newcommand{\nint}{\mathop{\mathrm{nint}}\nolimits}

\newcommand{\eps}{\varepsilon}
\newcommand{\eqdistr}{\stackrel{d}{=}}

\newcommand{\todistr}{\overset{d}{\underset{n\to\infty}\longrightarrow}}

\newcommand{\ton}{\overset{}{\underset{n\to\infty}\longrightarrow}}

\newcommand{\ind}{\mathbbm{1}}

\newcommand{\dd}{{\rm d}}
\newcommand{\eee}{{\rm e}}

\newcommand{\stirling}[2]{\genfrac{[}{]}{0pt}{}{#1}{#2}}
\newcommand{\stirlingb}[2]{\genfrac{\{}{\}}{0pt}{}{#1}{#2}}

\theoremstyle{plain}
\newtheorem{theorem}{Theorem}[section]

\newtheorem{proposition}[theorem]{Proposition}

\theoremstyle{definition}

\theoremstyle{remark}
\newtheorem{remark}[theorem]{Remark}

\newcounter{step}[theorem]

\begin{document}

\author{Zakhar Kabluchko}
\address{Zakhar Kabluchko, Institut f\"ur Mathematische Statistik,
Universit\"at M\"unster,
Or\-l\'e\-ans--Ring 10,
48149 M\"unster, Germany}
\email{zakhar.kabluchko@uni-muenster.de}

\author{Alexander Marynych}
\address{Alexander Marynych, Faculty of Cybernetics, Taras Shevchenko National University of Kyiv, 01601 Kyiv, Ukraine}
\email{{marynych@unicyb.kiev.ua}}

\author{Henning Sulzbach}
\address{Henning Sulzbach,
School of Computer Science, McGill University,
3480 University Street
H3A 0E9 Montr\'eal, QC, Canada}
\email{{henning.sulzbach@gmail.com}}

\title[Mode and Edgeworth expansion for the Ewens distribution]
{Mode and Edgeworth expansion for the Ewens distribution and the Stirling numbers}


\thanks{}
\begin{abstract}
We provide asymptotic expansions for the Stirling numbers of the first kind and, more generally, the Ewens (or Karamata-Stirling) distribution. Based on these expansions, we obtain some new results on the asymptotic properties of the mode and the maximum of the Stirling numbers and the Ewens distribution.  For arbitrary $\theta>0$ and for all sufficiently large $n\in\N$, the unique maximum of the Ewens probability mass function
$$
\mathbb L_n(k) = \frac{\theta^k}{\theta(\theta+1)\ldots(\theta+n-1)} \stirling {n}{k}, \quad k=1,\ldots,n,
$$
is attained at $k= \left\lfloor \theta \log n + \frac{\theta \Gamma'(\theta)}{\Gamma(\theta)} - \frac 12\right\rfloor$
or  $k=\left\lceil \theta \log n + \frac{\theta \Gamma'(\theta)}{\Gamma(\theta)} + \frac 12\right\rceil$.
We prove that the mode is
$$
k=\left\lfloor \theta\log n  - \frac{\theta \Gamma'(\theta)}{\Gamma(\theta)}\right\rfloor
$$
for a set of $n$'s of asymptotic density $1$, yet this formula is not true for infinitely many $n$'s.
\end{abstract}

\maketitle

\section{Introduction and statement of results}
\subsection{Introduction}
The (unsigned) \textit{Stirling numbers} of the first kind $\stirling{n}{k}$ are defined, for $n\in\N$ and $1\leq k\leq n$, by the formula
\begin{equation}\label{eq:stirling_def}
x^{(n)} := x(x+1)\ldots (x+n-1)  = \sum_{k=1}^n \stirling{n}{k} x^k,\quad x\in\R.
\end{equation}
For $n\in\N$, a random variable $K_n(\theta)$ is said to have {the} {\it Ewens distribution}  with parameter $\theta>0$ if its probability mass function is given by the formula
$$
\P\{K_n(\theta)=k\}=\frac{\theta^k}{\theta^{(n)}}\stirling{n}{k},\quad k=1,\ldots,n.
$$
The same distribution was also called \textit{Karamata-Stirling law} in~\cite{bingham_tauberian}.
One can interpret $K_n(\theta)$ as the number of blocks in a random partition of $\{1,\ldots,n\}$ distributed according to the Ewens sampling formula, or, equivalently, the number of different alleles in the infinite alleles model, the number of tables in a Chinese restaurant process, or the number of colors in the Hoppe urn. The Ewens sampling formula  plays an important role in population genetics; see, e.g.,\
\cite{ewens_tavare} and~\cite[Section~1.3]{durrett_DNA}.  There is a distributional representation of $K_n(\theta)$ as a sum of independent random variables
\begin{equation}\label{eq:K_n_theta_rep}
K_n(\theta) \eqdistr \xi_1+\ldots+\xi_n, \text{ where } \xi_i\sim \text{Bern}(\theta/(\theta + i-1))
\end{equation}
and $\text{Bern}(p)$ denotes the Bernoulli distribution with parameter $p$. In the special case $\theta=1$, {classical results going back at least to  Feller \cite{feller1945} and R\'enyi \cite{renyi} state that} the random variable $K_n(1)$ has the same distribution as the number of cycles in a uniformly chosen random permutation of $n$ objects, or the number of records in a sample of $n$ i.i.d.\ variables from a continuous distribution; see {also}~\cite{pitman_book} for these and other properties. It follows easily from Lindeberg's theorem that $K_n(\theta)$ satisfies a central limit theorem of the form
$$
\frac{K_n(\theta) - \theta \log n}{\sqrt{\theta \log n}} \todistr \text{N}(0,1)
$$
known as Goncharov's CLT in the case $\theta=1$.

Asymptotic expansions, as $n\to\infty$, of the Stirling numbers $\stirling {n}{k}$ in various regions of $k$ were provided in~\cite{moser_wyman,temme,wilf,hwang_stirling,hwang_diss,louchard_first_kind}. Most notably, \citet[Theorem~2]{hwang_stirling} (see also~\cite[Theorem~14 on p.~108]{hwang_diss} for a more general result) gave an asymptotic expansion valid uniformly in the domain $2\leq k \leq \eta \log n$, for any fixed $\eta>0$. Louchard~\cite[Theorem~2.1]{louchard_first_kind} computed three non-trivial terms of the asymptotic expansion in the central regime $k= \log n + O(\sqrt {\log n})$ which is similar to the classical Edgeworth expansion in the central limit theorem.

In this short note we start by deriving a full Edgeworth expansion, as $n\to\infty$, for the sequence of probability mass functions $k\mapsto \P\{K_n(\theta)=k\}$ which is uniform both in $\theta\in [1/\eta, \eta]$ (where $\eta>1$) and in
$k\in \{1,\ldots,n\}$; see Theorem~\ref{theo:stirling_edgeworth}.  Our result is an application of the general Edgeworth expansion for deterministic or random profiles obtained recently in \cite{Kabluchko+Marynych+Sulzbach:2016}.
Using this asymptotic expansion we derive some new results on the mode and the maximum of the Ewens distribution. In the case $\theta=1$ the mode can be interpreted as the most probable number of cycles in a random permutation of $n$ objects and was investigated in the works of \citet{hammersley} (see also~\cite[pp.~216--225]{hammersley_restricted}) and \citet{erdoes_on_hammersley}.
Our results on the mode and the maximum will be stated in Theorems~\ref{theo:stirling_mode_maximum} and~\ref{theo:stirling_mode_further} below.


\subsection{Asymptotic expansion of the Ewens distribution}
Before stating our main result some notions have to be recalled. The (complete) \textit{Bell polynomials} $B_j(z_1,\ldots,z_j)$ are defined by the formal identity
\begin{equation}\label{eq:bell_poly_def1}
\exp\left\{ \sum_{j=1}^{\infty} \frac{x^j}{j!} z_j \right\}
=
\sum_{j=0}^{\infty} \frac{x^j}{j!} B_j(z_1,\ldots,z_j).
\end{equation}
Therefore $B_0=1$ and, for $j\in\N$,
\begin{equation}\label{eq:bell_poly_def}
B_j(z_1,\ldots,z_j) = \sum{}^{'}\frac{j!}{i_1!\ldots i_j!} \left(\frac{z_1}{1!}\right)^{i_1}\ldots  \left(\frac{z_j}{j!}\right)^{i_j},
\end{equation}
where the sum $\sum{}^{'}$ is taken over all $i_1,\ldots,i_j\in\N_0$ satisfying $1i_1+2i_2+\ldots +j i_j =j$. For example, the first three Bell polynomials are given by
\begin{equation}\label{eq:Bell_poly_first}
B_1(z_1)=z_1, \quad
B_2(z_1,z_2) = z_1^2 + z_2, \quad
B_3(z_1,z_2,z_3) = z_1^3+3z_1z_2+z_3.
\end{equation}
Further, we will use the ``probabilist'' \textit{Hermite polynomials} $\Herm_n(x)$ defined by:
\begin{equation}\label{eq:Herm}
\Herm_n(x)= \eee^{\frac 12 x^2} \left(-\frac{\dd}{\dd x}\right)^n \eee^{-\frac 12 x^2},\quad  n\in\N_0.
\end{equation}
The first few Hermite polynomials needed for the first three terms of the expansion are
\begin{align*}
&\Herm_0(x)=1,\quad \Herm_1(x)= x,\quad
\Herm_2(x)= x^2-1, \quad
\Herm_3(x) = x^3-3x,\\
&\Herm_4(x)= x^4-6x^2+3, \quad
\Herm_6(x)= x^6 - 15 x^4 + 45 x^2 - 15.
\end{align*}

\begin{theorem}\label{theo:stirling_edgeworth}
Fix $r\in \N_0$ and a compact subset $L\subset (0,\infty)$. Uniformly over $\theta\in L$ we have
$$
\lim_{n\to\infty}(\log n)^{\frac{r+1}{2}}\sup_{k=1,\ldots,n}
\left|\P\{K_n(\theta)=k\} - \frac{\eee^{-\frac 12 x_n^2(k,\theta)}}{\sqrt{2\pi \theta \log n}} \sum_{j=0}^r \frac{H_j(x_n(k,\theta))}{(\theta \log n)^{j/2}}
\right|
= 0.
$$
Here, $x_n(k,\theta) = \frac{k - \theta \log n}{\sqrt{\theta \log n}}$ and $H_j(x)$ is a polynomial of degree $3j$ given by
\begin{align}\label{def_G_ewens}
H_j(x) := H_j(x,\theta) = \frac {(-1)^j} {j!} \, \eee^{\frac 12 x^2} B_j(\widetilde {D_1},\ldots, \widetilde {D_j}) \eee^{-\frac 12 x^2},
\end{align}
where $B_j$ is the $j$-th Bell polynomial and $\widetilde{D_1}, \widetilde{D_2},\ldots$ are differential operators given by
\begin{align}\label{eq:D_def}
\widetilde{D_j} := \widetilde{D_j}(\theta)
=
\frac{1}{(j+1)(j+2)} \left(\frac{\dd}{\dd x}\right)^{j+2} + \widetilde{\chi_j} (0) \left(\frac{\dd}{\dd x}\right)^{j}
\end{align}
with $\widetilde{\chi_j} (\beta) =  - \left(\frac{\dd}{\dd \beta}\right)^j \log \Gamma(\theta\eee^{\beta})$ and $\Gamma$ denoting the Euler gamma function.
\end{theorem}
\begin{remark}\label{rem:H_i}
It follows from~\eqref{eq:Bell_poly_first}, \eqref{def_G_ewens} and \eqref{eq:D_def}
that the first three coefficients of the expansion are given by
\begin{align*}
H_0(x)
&=
1,\label{eq:G0_stirling}\\
H_1(x)
&=
- \frac{\Gamma'(\theta)}{\Gamma(\theta)} \theta x + \frac{1}{6} \Herm_3(x),\\
H_2(x)
&=
\left(\theta^2 \frac{\Gamma'^2(\theta)}{\Gamma^2(\theta)} -  \frac{\theta^2\Gamma''(\theta)+\theta\Gamma'(\theta)}{2\Gamma(\theta)} \right)\Herm_2(x)
+
\left(\frac{1}{24} - \frac{\Gamma'(\theta)}{6\Gamma(\theta)}\theta\right) \Herm_4(x)\\
&+
\frac{1}{72}\Herm_6(x).
\end{align*}
An expression for $\widetilde {\chi_j}(0)$ involving  polygamma functions and  Stirling numbers of the second kind will be given in~\eqref{eq:chi_tilde_explicit}. The tilde in $\widetilde {D_j}$ and $\widetilde{\chi_j}$ is needed to keep the notation consistent with the paper~\cite{Kabluchko+Marynych+Sulzbach:2016}.  It is easy to check that $H_j(-x) = (-1)^j H_j(x)$; see~\cite[Remark~2.4]{Kabluchko+Marynych+Sulzbach:2016}.

To compute $H_j(x)$ one can proceed as follows. First, express $\frac {1}{j!}B_j(\widetilde {D_{1}}, \ldots, \widetilde{D_j})$ as a polynomial in $D:= \frac{\dd}{\dd x}$ (and note that only even/odd powers of $D$ are present if $j$ is even/odd). Then replace each occurrence of $D^l$ by $\Herm_l(x)$; see~\eqref{eq:Herm} for justification.
\end{remark}

\begin{remark}
It is possible to choose the value of $\theta$ as a function of $k$. One natural choice is $\theta = 1$ which provides a full version of Louchard's expansion~\cite[Theorem~2.1]{louchard_first_kind} (although he used a slightly different normalization in his analogue of $x_n(k,1)$ and his term $-355 x^3/144$ should be replaced by $-47 x^3/144$). Another possible choice is $\theta = k/\log n$ (so that $x_n(k,\theta)=0$),  which gives a large-deviation-type expansion valid uniformly in the region $\eta^{-1} \log n < k < \eta \log n$, for fixed $\eta >1$  and $q\in\N_0$:
$$
\frac{(k/\log n)^k}{(k/\log n)^{(n)}} \stirling{n}{k} =  \frac{1}{\sqrt{2\pi k}}\sum_{s=0}^{q} \frac{H_{2s}(0, k/\log n)}{k^s}  + o\left(\frac 1 {(\log n)^{q+1}} \right).
$$
Observe that the terms with half-integer powers of $k$ are not present in the sum because $H_{2j+1}(0)=0$. Using the formula
$$
\frac{\Gamma(n+\theta)}{n!} = n^{\theta-1} \left(1 + O\left(\frac 1n \right)\right)
$$
yields the expansion
\begin{equation}\label{eq:large_dev_expansion}
\frac 1 {n!} \stirling{n}{k}   =    \frac{1}{ \Gamma (\theta)}   n^{\theta -\theta \log \theta- 1}
\left(\frac{1}{\sqrt{2\pi k}}
\sum_{s=0}^{q} \frac{H_{2s}(0, \theta)}{k^s}  + o\left(\frac 1 {(\log n)^{q+1}} \right)\right)
\end{equation}
valid as $n\to\infty$ uniformly over $k$ in the region $\theta = k/\log n \in (\eta^{-1},\eta)$. In this region, this expansion must be equivalent to Hwang's result~\cite[Theorem~2]{hwang_stirling}.  It is not easy to rigorously verify this equivalence by a direct comparison, but we checked using Mathematica~9 that the first three non-trivial terms coincide. Note a misprint in the formula for the remainder term $Z_{\mu}(m,n)$ in~\cite[Theorem~2]{hwang_stirling}: $(\log n)^m/ (m! n)$ should be replaced by $(\log n)/(mn)$.
Expansion~\eqref{eq:large_dev_expansion}  could be also deduced from the work of~\citet[Theorem 3.4]{feray_meliot_nikeghbali}.
\end{remark}
Taking sums over $k$ in Theorem~\ref{theo:stirling_edgeworth} and using the Euler-Maclaurin formula to approximate Riemann sums by integrals, one obtains that
\begin{multline*}
\P\left\{\frac{K_n(\theta) - \theta \log n}{\sqrt{\theta \log n}} \leq x\right\} = \Phi(x)
\\
+ \frac{\eee^{-x^2/2}}{\sqrt{2\pi\theta \log n}}\left(\frac 12 - \frac{x^2-1}{6} + \theta \frac{\Gamma'(\theta)}{\Gamma(\theta)}\right)
+ O\left(\frac 1 {\log n}\right),
\end{multline*}
uniformly in $x\in (\theta\log n)^{-1/2}(\Z-\theta \log n)$, where $\Phi(x)$ is the standard normal distribution function. The justification is the same as in~\cite[Proposition~2.5]{gruebel_kabluchko_BRW} and is therefore omitted. {\citet{Yamato:2013} recently stated a slightly incorrect version of this expansion missing the term $1/2$ which comes from the Euler-Maclaurin formula.}
Similarly, one can obtain further terms in the expansion of the distribution function of $(K_n(\theta)-\theta \log n) / \sqrt{\theta \log n}$.

\subsection{Mode and maximum of the Ewens distribution}
Theorem~\ref{theo:stirling_edgeworth} allows us to deduce various results on the \emph{mode} and the \emph{maximum} of the Ewens distribution. The mode is any value $k\in\{1,\ldots,n\}$ maximizing $\P\{K_n(\theta) = k\}$, while the maximum $M_n(\theta)$ is defined by
\begin{equation*}
M_n(\theta) = \max_{1\leq k\leq n} \P\{K_n(\theta)=k\}.
\end{equation*}
Let us denote the least mode by $u_n(\theta)$.
In this context, it is important to note that, for all $\theta > 0$,  the function $k \mapsto \P\{K_n(\theta)=k\}$ is log-concave  by a theorem attributed to Newton~\cite{hammersley,sibuya}, and
\begin{multline}\label{eq:unimodal}
\P\{K_n(\theta) = 1\} < \ldots < \P\{K_n(\theta) = u_n(\theta)\}
\\
\geq   \P\{K_n(\theta) = u_n(\theta)+1\} > \ldots > \P\{K_n(\theta) = n\}.
\end{multline}
In particular, there are at most two modes.
For $\theta=1$, Erd\"{o}s~\cite{erdoes_on_hammersley}, proving a conjecture of Hammersley~\cite{hammersley}, showed that the mode is unique for all $n\geq 3$. {By~\eqref{eq:unimodal}}, uniqueness also holds for irrational $\theta$; however, for rational $\theta$ the mode need not be unique since for example
$$
\frac 23 \stirling{3}{1} = \left(\frac 23\right)^2 \stirling{3}{2} > \left(\frac 23\right)^3 \stirling{3}{3}.
$$

\begin{theorem}\label{theo:stirling_mode_maximum}
Fix $\theta>0$. There exists $N_1\in\N$ such that for $n\geq N_1$, the mode $u_n(\theta)$ of the Ewens distribution with parameter $\theta$ is unique and equals one of the numbers $\lfloor u_n^*(\theta)\rfloor$ or $\lceil u_n^*(\theta)\rceil$, where
\begin{equation}\label{eq:u_n_star_def}
u_n^*(\theta) = \theta\log n - \frac{\theta \Gamma'(\theta)}{\Gamma(\theta)} - \frac 12
\end{equation}
and $\lfloor \cdot \rfloor$, $\lceil \cdot \rceil$ denote the floor and the ceiling functions, respectively.
Write $\delta_n(\theta) := \min_{k\in\Z} |u_n^{*}(\theta) - k|$. For the maximum $M_n(\theta)$, we have
$$
 \sqrt{2\pi \theta \log n} \; M_n(\theta) = 1 + \frac{\theta (\log \Gamma)'(\theta) + \theta^2 (\log \Gamma)''(\theta) +1/12 - \delta_n^{2}(\theta) }{2\theta\log n} + o\left(\frac 1 {\log n}\right).
$$
\end{theorem}
In the case $\theta=1$,  related results for the mode were derived by \citet{hammersley} and \citet{erdoes_on_hammersley}, see also \citet{cramer} for statistical applications and~\citet{mezoe} for a generalization and an overview.
Theorem~\ref{theo:stirling_mode_maximum} states that the  mode is one of the numbers $\lfloor \log n + \gamma -\frac 12\rfloor$ or $\lceil \log n + \gamma -\frac 12\rceil$, for sufficiently large $n$. In fact, this holds for all $n\in\N$.
\begin{proposition}\label{prop:mode_stirling}
$
 u_n(1) \in \{\lfloor \log n + \gamma -\frac 12\rfloor, \lceil \log n + \gamma -\frac 12 \rceil \}
$
for all $n\in \N$.
\end{proposition}
The proof uses the following formula of~\citet{hammersley}:
\begin{equation}\label{eq:hammersley_mode_1}
u_n(1) = \left\lfloor  \log n + \gamma + \frac{\zeta(2)-\zeta(3)}{\log n + \gamma - \frac 32} + \frac{h(n)}{(\log n + \gamma - \frac 32)^2}\right\rfloor,
\end{equation}
for some  $-1.098011 < h(n) < 1.430089$; see also~\cite[Section~5.7.9]{hwang_diss} for a related formula.
\citet{erdoes_on_hammersley} observed that for $n>189$ Hammersley's formula implies that  the mode is one of the numbers $\lfloor \log (n-1) + \frac 12\rfloor $ or $\lfloor \log (n-1)+1\rfloor $. Note that his $\Sigma_{n,s}$ equals $\stirling{n+1}{n+1-s}$ and his $n-f(n)$ is $u_{n+1}(1)-1$ in our notation.  

The next theorem provides more precise information about the behavior of the mode.  Recall that a set $A \subset \N$ has \emph{asymptotic density} $\alpha \in [0,1]$ if
\begin{equation}\label{def:asymp_freq}
\lim_{n\to\infty} \frac {\# (A \cap \{1, \ldots, n\})}{n} = \alpha.
\end{equation}
For $x\in\R$, denote by $\{x\} = x - \lfloor x \rfloor$ the fractional part of $x$. Let $\nint(x)$ be the integer closest to $x$ (if $\{x\}=1/2$, we agree to take $\nint (x) = \lfloor x \rfloor$). That is,
$$
\nint(x):=\argmin_{k\in\Z}|x-k| = \left\lfloor x+\frac 12\right\rfloor.
$$

\begin{theorem}\label{theo:stirling_mode_further}
Fix $\theta>0$. The mode $u_n(\theta)$ of the Ewens distribution with parameter $\theta$ has the following properties:

\vspace*{2mm}
\noindent
\emph{(i)}
%
there exists a sufficiently large constant $C_0>0$  such that for all $n\in\N$ satisfying
$$
\left|\{u_n^*(\theta)\} - \frac 12\right| > \frac {C_0} {\log n},
$$
the mode $u_n(\theta)$ equals
$$
\nint(u_n^{\ast}(\theta)) = \left \lfloor  \theta\log n  - \frac{\theta \Gamma'(\theta)}{\Gamma(\theta)}\right\rfloor;
$$

\vspace*{2mm}
\noindent
\emph{(ii)}
there are arbitrarily long intervals of consecutive $n$'s for which $u_n(\theta)=\lceil u_n^{\ast}(\theta)\rceil$; similarly, there are arbitrarily long intervals of consecutive $n$'s for which $u_n(\theta)=\lfloor u_n^{\ast}(\theta)\rfloor$;

\vspace*{2mm}
\noindent
\emph{(iii)}
the set of $n\in\N$ such that $u_n(\theta)=\nint(u_n^{\ast}(\theta))$ has asymptotic density one;

\vspace*{2mm}
\noindent
\emph{(iv)}
there are infinitely many $n\in\N$ such that $u_n(\theta)\neq \nint(u_n^{\ast}(\theta))$.
\end{theorem}

The proof of part (iv) uses five terms in the Edgeworth expansion, where the first two terms influence the form of $u_n^*(\theta)$, while the remaining terms are needed for technical reasons. The idea is that the formula $u_n(\theta) = \nint(u_n^{\ast}(\theta))$ becomes wrong if the fractional part of $u_n^{*}(\theta)$ is slightly below $\frac 12$, so that higher order terms in the Edgeworth expansion decide which of the values $\lfloor u_n^*(\theta)\rfloor$ and $\lceil u_n^*(\theta)\rceil$ is the mode.    Using even more terms in the expansion, it is possible to replace $u_n^*(\theta)$ by some more complicated expressions involving higher-order corrections in inverse powers of $\theta \log n$ (see~\cite[Section~5.7.9]{hwang_diss}), but it seems that there is no formula of the form
$$
u_n(1) = \nint \left(\log n + a_0 + \frac {a_1}{\log n} + \ldots + \frac{a_r}{(\log n)^r}\right)
$$
which is valid for all sufficiently large $n$.

Finally, we would like to mention that one can easily obtain counterparts of the above results for the $B$- and $D$-analogues of Stirling numbers of the first kind. These are defined as the coefficients of $(x+1)(x+3)\ldots (x+2n-1)$ and $((x+1)(x+3)\ldots (x+2n-3)) (x+n-1)$, respectively, and appeared for example in~\cite{kabl_vys_zaporozhets_weyl_chambers}.

\section{Proofs}
\begin{proof}[Proof of Theorem \ref{theo:stirling_edgeworth}] The proof follows from the general Edgeworth expansion for random or deterministic profiles provided by Theorem 2.1 in \cite{Kabluchko+Marynych+Sulzbach:2016}. We consider the sequence of ``profiles''
\begin{align*}
\bL_n(k):=\P\{K_n(\theta)=k\} = \frac{\theta^k}{\theta^{(n)}}\stirling{n}{k} \ind_{k \in \{1, \ldots, n\}},
\end{align*}
and define
$$
w_n:=\theta\log n,\quad \varphi(\beta):=\eee^{\beta}-1,
\quad
(\beta_{-},\beta_{+})=\R,
\quad \fD=\{z\in\C\colon |\Im z| <\pi \}.
$$
In order to apply Theorem 2.1 in \cite{Kabluchko+Marynych+Sulzbach:2016}, we need to check the conditions A1--A4 given in the beginning of Section 2 of the cited paper. Note that
\begin{align*}
W_n(\beta)&:=\eee^{-\varphi(\beta)w_n}\sum_{k\in\Z}\eee^{\beta k}\bL_n(k)=n^{-\theta(\eee^{\beta}-1)}\sum_{k=1}^{n}\eee^{\beta k}\frac{\theta^k}{\theta^{(n)}}\stirling{n}{k}\\
&= n^{-\theta(\eee^{\beta}-1)}\frac{(\theta\eee^{\beta})^{(n)}}{\theta^{(n)}}
 = n^{-\theta (\eee^{\beta}-1)} \frac{\Gamma(\theta\eee^\beta  + n) \Gamma(\theta)}{\Gamma(\theta\eee^\beta)\Gamma(\theta+n)}
\ton
\frac{\Gamma(\theta)}{\Gamma(\theta\eee^\beta)}=:W_{\infty}(\beta)
\end{align*}
locally uniformly in $\beta\in\fD$ with speed polynomial in $n^{-1}$. Hence conditions A1--A3 are satisfied. In order to check condition A4 it is enough to show that for every $a>0$, $r\in\N$ and every compact subset $K_1$ of $\R$
\begin{equation}\label{proof_thm_main2}
\sup_{\beta \in K_1} \sup_{a \leq u \leq \pi}\left(n^{-\theta (\eee^\beta -1)} \left|\frac{\Gamma(\theta\eee^{\beta + iu}  + n)\Gamma(\theta)}{\Gamma(\theta+n)\Gamma(\theta\eee^{\beta+iu})}\right|\right)=o(\log^{-r} n),\quad n\to\infty.
\end{equation}
But this easily follows from
\begin{align*}
\lefteqn{\sup_{\beta \in K_1} \left(n^{-\theta (\eee^\beta -1)} \sup_{a \leq u \leq \pi} \left|\frac{\Gamma(\theta\eee^{\beta + iu}  + n)\Gamma(\theta)}{\Gamma(\theta+n)\Gamma(\theta\eee^{\beta+iu})}\right|\right)}\\
&\leq
C\sup_{\beta \in K_1} \left(n^{-\theta (\eee^\beta -1)} \sup_{a \leq u \leq \pi} \left|\frac{\Gamma(\theta\eee^{\beta + iu}  + n)}{\Gamma(\theta+n)}\right|\right)\leq C_1 \sup_{\beta \in K_1} n^{\theta \eee^{\beta}(\cos a-1)},
\end{align*}
with constants $C,C_1$ depending on $K_1$, $\theta$ and $a$. Therefore, Theorem 2.1 of \cite{Kabluchko+Marynych+Sulzbach:2016} is applicable for the Ewens distribution with arbitrary fixed $\theta>0$. In particular, for $\theta=1$, we obtain
\begin{align}\label{proof_thm_main4}
(\log n)^{\frac{r+1}{2}} \sup_{\beta\in K} \sup_{1 \leq k \leq n} \left| \frac{\Gamma(\eee^{\beta}) \eee^{\beta k}}{n^{\eee^{\beta}-1} n!}\stirling{n}{k} - \frac{\eee^{-\frac 12 x_n^2(k,\eee^{\beta})}}{\sqrt{2 \pi \eee^{\beta} \log n}}
\sum_{j=0}^r \frac{G_j(x_n(k,\eee^{\beta}); \beta)}{(\log n)^{j/2}} \right| \ton 0,
\end{align}
where $K$ is a compact subset of $\R$ and the polynomials $G_0, G_1, \ldots$ are defined as in Theorem 2.1 of \cite{Kabluchko+Marynych+Sulzbach:2016}: for $j \in \N_0$, we have
\begin{equation}\label{eq:def_G}
G_j(x; \beta) = \frac{(-1)^j}{j!} \eee^{\frac 1 2 x^2} B_j(D_1, \ldots, D_j) \eee^{-\frac 1 2 x^2}
\end{equation}
with the differential operators
\begin{equation}\label{eq:D_def_no_tilde}
D_j := D_j(\beta) = \eee^{-\frac 1 2 \beta j}\left( \frac{1}{(j+1)(j+2)} \left ( \frac{\dd}{\dd x}\right)^{j+2} + \chi_j(\beta) \left(\frac{\dd}{\dd x}\right)^{j}\right),
\end{equation}
where
$$\chi_j(\beta) = - \left(\frac{\dd}{\dd \beta}\right)^{j} \log \Gamma (\eee^\beta).$$
 Now, if $L \subseteq (0,\infty)$ is compact, then $K:=\log L$ is compact in $\R$. Applying \eqref{proof_thm_main4} with $K= \log L$ and $\beta=\log \theta\in K$ we obtain
\begin{align*}
(\log n)^{\frac{r+1}{2}} \sup_{\theta\in L} \sup_{1 \leq k \leq n} \left| \frac{\Gamma(\theta) \theta^k}{n^{\theta-1} n!}\stirling{n}{k} - \frac{\eee^{-\frac 12 x_n^2(k,\theta)}}{\sqrt{2 \pi \theta \log n}}
\sum_{j=0}^r \frac{G_j(x_n(k,\theta); \log \theta)}{(\log n)^{j/2}} \right| \ton 0.
\end{align*}
By Stirling's formula, uniformly in $\theta \in L$, $n \in\N$ and  $1 \leq k \leq n$, we have
\begin{align}
\frac{\Gamma(\theta) \theta^k}{n^{\theta-1} n!}\stirling{n}{k} = \frac{\theta^k}{\theta^{(n)}} \stirling{n}{k} (1 + O(n^{-1}))=\frac{\theta^k}{\theta^{(n)}} \stirling{n}{k}+ O(n^{-1}).
\end{align}
We conclude the proof by noting that $G_j(x; \log \theta)=\theta^{-j/2}H_j(x)$  which follows directly from $\widetilde {\chi_j}(0) = \chi_j(\log \theta)$. Indeed, by comparing~\eqref{eq:D_def} and~\eqref{eq:D_def_no_tilde}, we obtain
$$
D_j(\log \theta) = \theta ^{-j/2} \widetilde {D_j}(\theta),
$$
which implies that
$$
B_j(D_1(\log \theta),\ldots,D_j(\log \theta)) = \theta^{-j/2} B_j(\widetilde {D_1}(\theta), \ldots, \widetilde {D_j}(\theta))
$$
since $B_j(z_1,\ldots,z_j)$ is a sum of terms of the form $c \cdot z_1^{i_1}z_2^{i_2} \ldots z_j^{i_j}$ with $1i_1+2i_2+ \ldots + ji_j =j$; see~\eqref{eq:bell_poly_def}. Comparing~\eqref{def_G_ewens} and~\eqref{eq:def_G}, we obtain the required identity $G_j(x; \log \theta)=\theta^{-j/2}H_j(x)$.

To see that $\widetilde {\chi_j}(0) = \chi_j(\log \theta)$, one can easily show by induction over $j \geq 1$ that both
\begin{equation*} 
\chi_j(\beta) =  - \sum_{\ell = 1}^{j} \stirlingb{j}{\ell} \psi^{(\ell -1)}(\eee^\beta) \eee^{ \ell \beta},
\end{equation*}
and
\begin{equation}\label{eq:chi_tilde_explicit}
\widetilde{\chi_j}(\beta) =  - \sum_{\ell = 1}^{j} \stirlingb{j}{\ell} \psi^{(\ell-1)}(\theta \eee^\beta) (\theta\eee^{\beta})^{\ell}.
\end{equation}
Here $\psi^{(j)}(x) = (\log \Gamma(x))^{(j+1)}$ denotes the polygamma function and $\stirlingb{n}{k}$ is the Stirling number of the second kind satisfying the recurrence
$$
\stirlingb{n+1}{k} = \stirlingb{n}{k-1} + k \stirlingb{n}{k}, \quad 1 \leq k \leq n, \;\; n \in\N,
$$
with initial conditions $\stirlingb{0}{0} = 1$, $\stirlingb{n}{0} = \stirlingb{0}{n} = 0$.
\end{proof}

\begin{proof}[Proof of Theorem \ref{theo:stirling_mode_maximum}]
It follows from Theorems~2.10 in \cite{Kabluchko+Marynych+Sulzbach:2016} that for sufficiently large $n$, the maximizers of the function $k\mapsto \P\{K_n(\theta) = k\}$ must be of the form $\lfloor  u_n^* \rfloor$ or $\lceil u_n^*\rceil$.

Next we prove that the maximizer is unique (for sufficiently large $n$) by following the method of \citet{erdoes_on_hammersley} who considered the case $\theta=1$.
{Due to~\eqref{eq:unimodal},} the uniqueness is evident if $\theta$ is irrational. Hence, we assume that $\theta = Q_1/Q_2$ is rational with $Q_1,Q_2$ being integer.
We have, by~\eqref{eq:stirling_def},
$$
\stirling{n}{k} = \sum_{1\leq a_1 <\ldots < a_{n-k} \leq n-1} a_1\ldots a_{n-k}.
$$
Put $k_n = \lceil u_n^*(\theta)\rceil = \theta \log n + O(1)$ as $n\to\infty$.
{By~\eqref{eq:unimodal}, it is sufficient} to show that
\begin{equation}\label{eq:mode_unique_have_to_show}
\theta^{k_n} \stirling {n}{k_n} \neq \theta^{k_n-1} \stirling {n}{k_n-1}.
\end{equation}
By the prime number theorem with an appropriate error term, see~\cite{erdoes_on_hammersley}, for all sufficiently large $n$ there is a prime number $p$ satisfying {$(n-1)/k_n < p < (n-1)/(k_n-1)$.}   Then,
$$
\stirling {n}{k_n}
\not \equiv  0 \pmod p,
\quad
\stirling {n}{k_n - 1}
\equiv  0 \pmod p
$$
because in the representation of the former Stirling number all products except one are divisible by $p$, whereas in the latter all products are divisible by $p$; see~\cite{erdoes_on_hammersley}.   If $n$ is large, $p$ is not among the prime factors of $Q_1$ and $Q_2$, hence~\eqref{eq:mode_unique_have_to_show} follows and the mode of $K_n(\theta)$ is uniquely defined.
Finally, the formula for $M_n$ follows from Theorem~2.13 of~\cite{Kabluchko+Marynych+Sulzbach:2016}.
\end{proof}

\begin{proof}[Proof of Proposition~\ref{prop:mode_stirling}]
It was shown by Hammersley~\cite{hammersley} that
\begin{equation}\label{eq:hammersley_mode}
u_n(1) = \left\lfloor  \log n + \gamma + \frac{\zeta(2)-\zeta(3)}{\log n + \gamma - \frac 32} + \frac{h(n)}{(\log n + \gamma - \frac 32)^2}\right\rfloor
\end{equation}
with some $-1.1 < h(n) < 1.44$. It is easy to check that
$$
\frac{\zeta(2)-\zeta(3)}{x} - \frac {1.1}{x^2} > -\frac 12 \text{ and }
 \frac{\zeta(2)-\zeta(3)}{x} + \frac {1.44}{x^2} <\frac 12
$$
for $x>2.5$. Hence, the proposition is true for $\log n + \gamma - \frac 32 > 2.5$, that is for $n\geq 31$. For $n=1,2,\ldots,30$ the statement is easy to verify using Mathematica~9.
\end{proof}

\begin{proof}[Proof of parts  (i) and (ii) of Theorem \ref{theo:stirling_mode_further}]
Part (i) follows essentially from Theorem~2.10 in~\cite{Kabluchko+Marynych+Sulzbach:2016} and its proof. Namely, it was shown in~\cite[Equation~(90)]{Kabluchko+Marynych+Sulzbach:2016}  that for $k= k(n)= u_n^*(\theta) + g \in \Z$ with $g=O(1)$ we have
$$
\sqrt{2\pi \theta \log n} \left(\P\{K_n(\theta) = k+1\} - \P\{K_n(\theta) = k\}\right)
=
-\frac {2g+1}{2\theta \log n} + o\left(\frac 1 {\log n} \right).
$$
The same relation, but with a better remainder term $O(\frac 1 {\log^2 n})$, follows from~\eqref{proof_iv1} which we shall prove below.
Taking $g=-\{u_n^*(\theta)\}$, so that $k= \lfloor u_n^*(\theta)\rfloor$ and $k+1=\lceil u_n^*(\theta)\rceil$, yields
\begin{multline*}
\P\{K_n(\theta) = \lceil u_n^*(\theta)\rceil\} - \P\{K_n(\theta) = \lfloor u_n^*(\theta)\rfloor\}
\\=
\frac 1 {\sqrt{2\pi \theta \log n}} \left(\frac {\{u_n^*(\theta)\}-\frac 12}{\theta \log n} + O\left(\frac 1 {\log^2 n} \right)\right).
\end{multline*}
It follows that there is a sufficiently large constant $C_0>0$ such that if $\{u_n^*(\theta)\} > \frac 12 + \frac {C_0}{\log n}$, then the right-hand side is positive and the mode equals $\lceil u_n^*(\theta)\rceil$.  Similarly, if $\{u_n^*(\theta)\} < \frac 12 - \frac {C_0}{\log n}$, then the right-hand side is negative and the mode equals $\lfloor u_n^*(\theta)\rfloor$.

The proof of part (ii) follows immediately from part (i) and the fact that, for every fixed $L>0$, we have  $\log (n+L)-\log n\to 0$ as $n\to\infty$. \end{proof}

\begin{proof}[Proof of part (iii) of Theorem \ref{theo:stirling_mode_further}]
In view of part (i) it suffices to show that
$$\limsup_{\eps \to 0} \limsup_{n \to \infty} \frac{\#\{1 \leq k  \leq n: \text{dist}(u_k^*(\theta), \Z + 1/2) < \eps \} }{n} = 0,$$
which, in turn, follows from the fact
\begin{equation}\label{proof_iii}
\limsup_{\eps \to 0} \limsup_{n \to \infty} \frac{\#\{1 \leq k  \leq n: \text{dist}(\log k, \alpha \Z + \beta) < \eps \}  }{n} = 0,
\end{equation}
for all $\alpha>0$ and $\beta \in \R$. Equation~\eqref{proof_iii} would be true if the sequence of fractional parts of $\alpha^{-1}\log k$, $k\in\N$, was uniformly distributed on $[0,1]$. However, the latter claim is unfortunately not true; see~\cite[Examples~2.4 and~2.5 on pp.~8--9]{kuipers_niederreiter_book}.   Let us prove \eqref{proof_iii}. We have, assuming that $\varepsilon<\alpha/2$,
\begin{align*}
\#\{1 \leq k  \leq n: \text{dist}(\log k,  & \alpha \Z + \beta) < \eps \}
=\sum_{k=1}^{n} \# \{ j \in \Z : \text{dist}(\log k, \alpha j + \beta) < \eps\}\\
& = \sum_{j\in\Z}\#\{1 \leq k  \leq n: \eee^{\alpha j+\beta-\eps} < k < \eee^{\alpha j+\beta+\eps} \} \\
& \leq \sum_{j\in\Z} \# \{k \in \N : \eee^{\alpha j+\beta-\eps}\vee 1 \leq k \leq \eee^{\alpha j+\beta+\eps}\wedge n \}.
\end{align*}
The summand on the right-hand side is the number of integers in the interval $[\eee^{\alpha j+\beta-\eps}\vee 1,\,\eee^{\alpha j+\beta+\eps}\wedge n]$ (which is empty if either $\eee^{\alpha j+\beta-\eps}>n$ or $\eee^{\alpha j+\beta+\eps}<1$) and hence is bounded from above by
$\left(\eee^{\alpha j+\beta+\eps}\wedge n - \eee^{\alpha j+\beta-\eps}\vee 1+1\right)_{+}$. Therefore,
$$
\#\{1 \leq k  \leq n: \text{dist}(\log k, \alpha \Z + \beta) < \eps \} \leq \sum_{j\in\Z}\left(\eee^{\alpha j+\beta+\eps}\wedge n - \eee^{\alpha j+\beta-\eps}\vee 1+1\right)_{+}.
$$
Further,
\begin{align*}
&\sum_{j\in\Z}\left(\eee^{\alpha j+\beta+\eps}\wedge n - \eee^{\alpha j+\beta-\eps}\vee 1+1\right)_{+}\\
&=\sum_{j\in\Z}\eee^{\alpha j+\beta+\eps}\ind_{\{\alpha j+\beta+\varepsilon<0\}}+\sum_{j\in\Z}\eee^{\alpha j+\beta+\eps}\ind_{\{\alpha j+\beta-\varepsilon<0,0\leq \alpha j+\beta+\varepsilon<\log n\}}\\
&+\sum_{j\in\Z}n\ind_{\{\alpha j+\beta-\varepsilon<0,\log n\leq \alpha j+\beta+\varepsilon\}}\\
&+\sum_{j\in\Z}\left(\eee^{\alpha j+\beta+\eps} - \eee^{\alpha j+\beta-\eps}+1\right)\ind_{\{\alpha j+\beta-\varepsilon\geq 0,\alpha j+\beta+\varepsilon<\log n\}}\\
&+\sum_{j\in\Z}\left(n - \eee^{\alpha j+\beta-\eps}+1\right)_{+}\ind_{\{\alpha j+\beta-\varepsilon\geq 0,\log n\leq \alpha j+\beta+\varepsilon\}}.
\end{align*}
Note that the first series converges, the second contains at most one summand since we assume $\varepsilon<\alpha/2$, and the third vanishes for $n$ large enough. It can be checked that
\begin{align*}
\sum_{j\in\Z}\left(\eee^{\alpha j+\beta+\eps} - \eee^{\alpha j+\beta-\eps}+1\right)\ind_{\{\alpha j+\beta-\varepsilon\geq 0,\alpha j+\beta+\varepsilon<\log n\}}\leq C(\alpha,\beta)(\eee^{\beta+\eps}-\eee^{\beta-\eps})n
\end{align*}
with $C(\alpha,\beta)$ being an absolute constant, and that for $n$ sufficiently large
\begin{align*}
\sum_{j\in\Z}\left(n - \eee^{\alpha j+\beta-\eps}+1\right)_{+}\ind_{\{\alpha j+\beta-\varepsilon\geq 0,\log n\leq \alpha j+\beta+\varepsilon\}}\leq n(1-\eee^{-2\eps})+1.
\end{align*}
Putting pieces together gives \eqref{proof_iii}.
\end{proof}

\begin{proof}[Proof of part (iv) of Theorem \ref{theo:stirling_mode_further}]
Recall the notation $w_n=\theta\log n$ and $x_n(k) = x_n(k,\theta)=(k-w_n)/\sqrt{w_n}$. Using Theorem~\ref{theo:stirling_edgeworth} with $r=4$, we obtain
\begin{multline*}
\sqrt{2\pi w_n}\, \P\{K_n(\theta)=k\}=\eee^{-\frac 12 {x_n^2(k)}}\\
\times\left(1+\frac{H_1(x_n(k))}{w_n^{1/2}}+\frac{H_2(x_n(k))}{w_n}+\frac{H_3(x_n(k))}{w_n^{3/2}}+\frac{H_4(x_n(k))}{w_n^2}+o\left(\frac{1}{\log^2 n}\right)\right),
\end{multline*}
{as $n \to\infty$ uniformly in $1 \leq k \leq n$.}
Now let $k=\theta\log n + a$, where $a=O(1)$ as $n\to\infty$,  so that $x_n(k)=a/w_n^{1/2}$. We have
\begin{align*}
H_1(x_n(k))&=A_{11}(\theta)\frac{a}{w_n^{1/2}}+A_{12}(\theta)\frac{a^3}{w_n^{3/2}},\\
H_2(x_n(k))&=A_{21}(\theta)+A_{22}(\theta)\frac{a^2}{w_n}+o\left(\frac{1}{w_n}\right),\\
H_3(x_n(k))&=A_{31}(\theta)\frac{a}{w_n^{1/2}}+o\left(\frac{1}{w_n^{1/2}}\right),\\
H_4(x_n(k))&=A_{41}(\theta)+o(1),
\end{align*}
where $A_{11}(\theta),\ldots,A_{41}(\theta)$ are some polynomials in $\widetilde {\chi_1}(0), \widetilde {\chi_2}(0),\widetilde {\chi_3}(0)$ and $\widetilde{\chi_4}(0)$; see Remark~\ref{rem:H_i}.  Plugging these expressions into the asymptotic expansion above and using the {expansion} $\eee^y=1+y+y^2/2+o(y^2)$, as $y\to 0$, yields
$$
\sqrt{2\pi w_n}\,\P\{K_n(\theta)=k\}=1-\left(\frac{a^2}{2}-A_{11}(\theta)a-A_{21}(\theta)\right)\frac{1}{w_n}+\frac{P_{\theta}(a)}{w_n^2}+o\left(\frac{1}{\log^2 n}\right),
$$
where
$$
P_{\theta}(a):=\frac{1}{8}a^4+\left(A_{12}(\theta)-\frac{1}{2}A_{11}(\theta)\right)a^3+\left(A_{22}(\theta)-\frac{1}{2}A_{21}(\theta)\right)a^2+A_{31}(\theta)a + A_{41}(\theta).
$$
Now let us write $k=\theta\log n +a^{\ast} + g$, where $a^{\ast}:=A_{11}(\theta)=-\frac{\theta\Gamma'(\theta)}{\Gamma(\theta)}-\frac 12$, yielding
\begin{multline}\label{proof_iv1}
\sqrt{2\pi w_n}\,\P\{K_n(\theta)=k\}
=
1-\left(\frac{g^2-(a^{\ast})^2}{2}-A_{21}(\theta)\right)\frac{1}{w_n}\\
+\frac{P_{\theta}(a^{\ast}+g)}{w_n^2}+o\left(\frac{1}{\log^2 n}\right).
\end{multline}
We are interested in $g$ being either $\lfloor u_n^{\ast}(\theta)\rfloor -u_n^{\ast}(\theta)=:g_n^{\prime}$ or $\lceil u_n^{\ast}(\theta)\rceil -u_n^{\ast}(\theta)=:g_n^{\prime\prime}$. {Let $M$ be the set of natural numbers $n$ with $\{u_n^{\ast}(\theta)\}<1/2<\{u_{n+1}^{\ast}(\theta)\}$.}
Note that $M$ has infinitely many elements
because $\log n\to \infty$ and $\log (n+1)-\log n\to 0$. {In the remainder of the proof, we always consider $n \in M$.} Since $u^{\ast}_{n+1}(\theta)-u^{\ast}_n(\theta)=O(n^{-1})$, we have
\begin{align*}
g_n^{\prime}=-1/2+O(n^{-1}),
\quad 
g_n^{\prime\prime}=1/2+O(n^{-1}). 
\end{align*}
Putting $k = \lfloor u_n^{\ast}(\theta)\rfloor$ into \eqref{proof_iv1} yields
\begin{align*}
&\sqrt{2\pi w_n}\,\P\{K_n(\theta)=\lfloor u_n^{\ast}(\theta)\rfloor\}\\
&=1-\left(\frac{(g_n^{\prime})^2-(a^{\ast})^2}{2}-A_{21}(\theta)\right)\frac{1}{w_n}+\frac{P_{\theta}(a^{\ast}+g_n^{\prime})}{w_n^2}+o\left(\frac{1}{\log^2 n}\right)\\
&=1-\left(\frac{1-4(a^{\ast})^2}{8}-A_{21}(\theta)\right)\frac{1}{w_n}+\frac{P_{\theta}(a^{\ast}-1/2)}{w_n^2}+o\left(\frac{1}{\log^2 n}\right).
\end{align*}
Analogously, putting $k = \lceil u_n^{\ast}(\theta)\rceil$ gives
\begin{align*}
&\sqrt{2\pi w_n}\,\P\{K_n(\theta)=\lceil u_n^{\ast}(\theta)\rceil\}\\
&=1-\left(\frac{1-4(a^{\ast})^2}{8}-A_{21}(\theta)\right)\frac{1}{w_n}+\frac{P_{\theta}(a^{\ast}+1/2)}{w_n^2}+o\left(\frac{1}{\log^2 n}\right).
\end{align*}
For sufficiently large $n$ the mode $u_n(\theta)$ equals either $\lfloor u_n^{\ast}(\theta)\rfloor$ or $\lceil u_n^{\ast}(\theta)\rceil$  depending on the sign of
$$
s^{\ast}(\theta):=P_{\theta}(a^{\ast}+1/2)-P_{\theta}(a^{\ast}-1/2).
$$
In the following we shall show that $s^{\ast}(\theta)>0$, hence $u_n(\theta)=\lceil u_n^{\ast}(\theta)\rceil$, while $\nint (u_{n}(\theta)) = \lfloor u_n^{\ast}(\theta)\rfloor$, so that $u_{n}(\theta)\neq \nint(u_n^{\ast}(\theta))$.
With the aid of Mathematica~9, see~\cite{mathematica}, recalling the polygamma function $\psi^{(m)} (\theta) = (\log \Gamma(\theta))^{(m+1)}$, it can be checked that
$$
s^{\ast}(\theta)=\frac{\theta^2}{2}\left(2 \psi^{(1)}(\theta) +\theta\psi^{(2)}(\theta)	\right).
$$
Using a well-known formula for the polygamma function (see 6.4.10 in \cite{Abramowitz+Stegun:1964})
$$
\psi^{(m)} (\theta) = (\log \Gamma(\theta))^{(m+1)}=(-1)^{m+1}m!\sum_{k=0}^{\infty}\frac{1}{(\theta+k)^{m+1}},\quad -\theta\notin \N_0,\quad m \geq 1,
$$
we finally obtain
$$
s^{\ast}(\theta)=\theta^2\sum_{k=1}^{\infty}\frac{k}{(\theta+k)^3},\quad \theta>0,
$$
yielding positivity of $s^{\ast}(\theta)$ for all $\theta>0$. The proof of part (iv), as well as of the whole theorem, is complete.
\end{proof}
\begin{remark}
For $\theta=1$ we have $s^*(1) = \zeta(2) - \zeta(3)$, a term appearing in Hammersley's formula~\eqref{eq:hammersley_mode_1}. In fact, in the special case $\theta=1$ part (iv) could be deduced directly from~\eqref{eq:hammersley_mode_1}.
\end{remark}


\subsection*{Acknowledgement} The work of Alexander Marynych was supported by a Humboldt Research Fellowship of the Alexander von Humboldt Foundation. The work of Henning Sulzbach was supported by a Feodor Lynen Research Fellowship of the Alexander von Humboldt Foundation.

\bibliographystyle{plainnat}
\bibliography{stirling_numbers}

\end{document}